\documentclass[review]{elsarticle}

\usepackage[T1]{fontenc}
\usepackage{lmodern}
\usepackage[a4paper,margin=1in]{geometry}
\usepackage{amsmath,amssymb,amsthm,mathtools}
\usepackage{enumitem}
\usepackage{xcolor}
\usepackage{microtype}
\usepackage[colorlinks=true,linkcolor=blue!55!black,citecolor=blue!55!black,
            urlcolor=blue!55!black]{hyperref}

\setlist[itemize]{leftmargin=2em,itemsep=0.2em,topsep=0.3em}

\newtheorem{theorem}{Theorem}[section]
\newtheorem{proposition}[theorem]{Proposition}
\newtheorem{corollary}[theorem]{Corollary}
\newtheorem{lemma}[theorem]{Lemma}
\newtheorem*{maintheorem}{Main Theorem}
\theoremstyle{definition}
\newtheorem{definition}[theorem]{Definition}
\newtheorem{remark}[theorem]{Remark}

\newcommand{\ua}{\mathord{\uparrow}}
\newcommand{\da}{\mathord{\downarrow}}
\newcommand{\RO}{\operatorname{RO}}
\newcommand{\Iso}{\operatorname{Iso}}
\newcommand{\Int}[2]{\operatorname{int}_{#1}\!\left(#2\right)}
\newcommand{\Cl}[2]{\operatorname{cl}_{#1}\!\left(#2\right)}
\newcommand{\Pow}{\mathcal P}
\newcommand{\Scott}{\Sigma}
\newcommand{\cl}{\operatorname{cl}}

\begin{document}
\begin{frontmatter}

\title{\bfseries Sobriety of Regular Open Algebras in
Second-Countable $T_3$ Spaces}
\author[1]{Xiaoyong Xi}
\author[2,3]{Chong Shen\corref{cor1}}
\ead{shenchong0520@163.com}
\cortext[cor1]{Corresponding author.}
\author[4]{Dongsheng Zhao}

\address[1]{School of Mathematics and Statistics,	Yancheng Teachers University, Yancheng, Jiangsu,  China}
\address[2]{School of Mathematical Sciences,	Beijing University of Posts and Telecommunications,  Beijing, 
	China}
\address[3] {Key Laboratory of Mathematics and Information Networks (Beijing University of
	Posts and Telecommunications), Ministry of Education, China}
\address[4]{Mathematics and Mathematics Education, National Institute of Education,
	Nanyang Technological University,  1 Nanyang Walk, Singapore}
\date{}

\begin{abstract}
	For a topological space $X$, let $\RO(X)$ be the complete Boolean algebra
	of regular open subsets of $X$, ordered by inclusion. We prove that, for
	every second-countable $T_3$ space $X$, the Scott space $\Scott\RO(X)$ is
	sober if and only if $\Iso(X)$ is dense in $X$. Consequently, $\Sigma\RO(\mathbb R^n)$ is not sober for every
	positive integer $n$. In particular, $\Sigma\RO(\mathbb R)$ is not
	sober, which provides an answer to an open problem concerning the
	sobriety of complete Boolean algebras.
	This characterization also yields a systematic way to
	obtain more natural examples of complete lattices whose Scott spaces are
	non-sober.
\end{abstract}
\end{frontmatter}

\section{Introduction}

For a poset $P$, the Scott topology is the topology whose open sets are
the upper sets inaccessible by existing directed suprema.  Scott spaces
are central in domain theory, where sobriety permits the recovery of
points from irreducible closed sets; see Scott's foundational paper
\cite{Scott1972} and the standard references
\cite{GierzEtAl2003,GoubaultLarrecq2013}.  Continuous dcpos have sober
Scott spaces, but Scott sobriety fails in general.  Johnstone constructed
the first non-sober dcpo \cite{Johnstone1981}, and Isbell subsequently
obtained a non-sober complete lattice \cite{Isbell1982}.  Subsequent work gave a complete Heyting algebra with a non-sober Scott
space \cite{XuXiZhao2021}, as well as a countable complete lattice with a
non-sober Scott space \cite{MiaoEtAl2023}.

A particularly natural unresolved test case has been the complete Boolean
algebra
\[
B=\RO(\mathbb R)
\]
of regular open subsets of the real line, ordered by inclusion. Ern\'{e} and Gatzke showed that $\Sigma B$ is not a topological
join-semilattice \cite{ErneGatzke1985}.  As recorded by Xu and Zhao
\cite[Question~3.6]{XuZhao2021}, Ern\'{e}  later asked whether the Scott
space $\Sigma B$ is sober; see also
\cite[Question~3.6]{Xu2023Recent}.  The more general question whether
there exists any complete Boolean algebra with a non-sober Scott space was
recorded alongside it \cite[Question~3.7]{XuZhao2021,Xu2023Recent}.

The aim of the present paper is to answer these questions and place the
phenomenon in a general topological framework. For every second-countable
$T_3$ space $X$, we characterize precisely when the Scott space of the
regular open algebra $\RO(X)$ is sober.

\begin{maintheorem}
Let $X$ be a second-countable $T_3$ space.  Then the following statements
are equivalent:
\begin{enumerate}[label=\textup{(\arabic*)},leftmargin=2.8em]
  \item $\Scott\RO(X)$ is sober;
  \item the complete Boolean algebra $\RO(X)$ is atomic;
  \item $\Cl{X}{\Iso(X)}=X$.
\end{enumerate}
\end{maintheorem}
Consequently, for every positive integer $n$, neither
$\Scott\RO(\mathbb R^n)$ nor $\Scott\RO(2^{\mathbb N})$ is sober, where
$\mathbb R^n$ denotes the $n$-dimensional Euclidean space and
$2^{\mathbb N}$ denotes the Cantor space.  
In particular, the case $n=1$ therefore provides a negative answer to
Question~3.6 of Xu and Zhao~\cite{XuZhao2021}; see also
Question~3.6 in~\cite{Xu2023Recent}.
 Moreover, the characterization
provides a systematic way to obtain further natural examples of complete
lattices whose Scott spaces are non-sober.

\section{Preliminaries}

We collect some basic concepts and results that will be used later.  For more
details, we refer the reader to~\cite{GierzEtAl2003,GoubaultLarrecq2013}.

Throughout the paper, $\mathbb N=\{0,1,2,\ldots\}$. For each positive
integer $n$, let $\mathbb R^n$ denote the $n$-dimensional Euclidean space;
in particular, $\mathbb R^1=\mathbb R$.

Let $P$ be a poset.
For a subset $A$ of  $P$, we denote
$\ua A=\{y\in P: \exists x\in A, x\leq y\}$ and $\da A=\{y\in P:\exists x\in A, y\leq x\}$, respectively.
For each $x\in P$, we simply write $\ua x$ and $\da x$ for $\ua\{x\}$ and  $\da \{x\}$, respectively. A subset $A$ of $P$ is called a \emph{lower} (resp., an \emph{upper}) \emph{set} if $A=\da A$ (resp., $A=\ua A$).
An element $x$ is \emph{maximal} in $A\subseteq P$ if
$A\cap\ua x=\{x\}$.

A nonempty subset $D$ of $P$ is called \emph{directed} if for any $x,y\in D$, there exists $z\in D$ such that $x\leq z$ and $y\leq z$.
We call $P$ a \emph{dcpo} (\emph{directed-complete poset}) if for any directed subset $D$ of $P$, $\bigvee D$ exists.
A subset $U$ of  $P$ is \emph{Scott open} if
(i) $U=\mathord{\uparrow}U$ and (ii) for each directed subset $D$ of $P$ with $\bigvee D$ existing, $\bigvee D\in U$ implies $D\cap
U\neq\emptyset$. All Scott open subsets of $P$ form a topology on $P$,
called the \emph{Scott topology} and denoted by $\sigma(P)$. The space $\Sigma P=(P,\sigma(P))$ is called the
\emph{Scott space} of $P$.

We shall use the following standard criteria.  If $P$ is a dcpo, then a
subset of $P$ is Scott closed if and only if it is a lower set closed under
directed suprema.  A map between dcpos is Scott continuous if and only if it
is monotone and preserves directed suprema.

\begin{definition}\label{def:basic-topology}
	Let $X$ be a topological space.
	\begin{enumerate}[label=\textup{(\roman*)},leftmargin=2.8em]
		\item A subset $A\subseteq X$ is \emph{dense} in $X$ if
		$\Cl{X}{A}=X$. Equivalently, every nonempty open subset of $X$
		meets $A$.

		\item A family $\mathcal B$ of open subsets of $X$ is a \emph{base}
		for the topology of $X$ if, whenever $x\in U$ and $U$ is open in
		$X$, there exists $B\in\mathcal B$ such that
		$x\in B\subseteq U$. The space $X$ is \emph{second-countable} if
		its topology has a countable base.

		\item The space $X$ is a $T_1$ space if every singleton is closed.
		It is \emph{regular} if, whenever $x\in X$ and $F\subseteq X$ is
		closed with $x\notin F$, there exist disjoint open sets $V$ and $W$
		such that $x\in V$ and $F\subseteq W$. Equivalently, for every open
		set $U$ and every $x\in U$, there exists an open set $V$ such that
		\[
			x\in V\subseteq\Cl{X}{V}\subseteq U.
		\]
		The space $X$ is a $T_3$ space if it is both regular and $T_1$.

		\item A point $x\in X$ is \emph{isolated} if $\{x\}$ is open in
		$X$. We write $\Iso(X)$ for the set of all isolated points of $X$.
	\end{enumerate}
\end{definition}

\begin{definition}
	A nonempty subset $A$ of a space $X$ is \emph{irreducible} if for any closed sets $C_1, C_2$ of $X$, $A\subseteq C_1\cup C_2$ implies $A\subseteq C_1$ or $A\subseteq C_2$.
	A $T_0$ space $X$ is called \emph{sober} if for any irreducible closed set $A$, $A=\cl_X(\{x\})$ for some $x\in X$.
  Equivalently, a nonempty subset $A$ is irreducible if, whenever open sets
  $U$ and $V$ both meet $A$, their intersection $U\cap V$ also meets $A$;
  see~\cite{GierzEtAl2003,GoubaultLarrecq2013}.
\end{definition}

	A \emph{retract} of a topological space $Y$ is a topological space $X$
such that there exist continuous maps
$s\colon X\to Y$ (the \emph{section}) and
$r\colon Y\to X$ (the \emph{retraction}) satisfying
$r\circ s=\operatorname{id}_X$.

\begin{remark}\label{rem:standard-facts}
	The following standard facts will be used repeatedly.
	\begin{enumerate}[label=\textup{(\arabic*)}]
		\item Every retract of a sober space is sober;
		see~\cite[Chapter~8]{GoubaultLarrecq2013}.
		\item
			For every
			$x\in P$, the closure of $\{x\}$ in $\Scott P$ is
			$\mathord\downarrow x$. Consequently, a Scott-closed subset of a complete lattice
			which is the closure of a point has a greatest element.
		\item For every set $A$, the Scott space of the powerset $\Pow(A)$ under the inclusion order
		is sober.  Indeed, $\Pow(A)$ is an algebraic  dcpo,
		and the Scott space of every algebraic dcpo is sober;
		see~\cite[Chapter~II]{GierzEtAl2003}.
	\end{enumerate}
\end{remark}

A Boolean algebra $B$ is \emph{complete} if every subset of $B$ has a join
and a meet.  We write $0,1,\wedge,\vee$, and $\neg$ for its least element,
greatest element, meet, join, and complementation, respectively.  It is
\emph{nontrivial} if $0\neq1$. 
In every complete Boolean algebra, the following infinitary distributive
identities hold for all $b\in B$ and $A\subseteq B$:
\begin{align}
	b\wedge\bigvee A &= \bigvee_{a\in A}(b\wedge a),
	\label{eq:meet-distributes}\\
	b\vee\bigwedge A &= \bigwedge_{a\in A}(b\vee a).
	\label{eq:join-distributes}
\end{align}
In particular, maps obtained by taking the meet with a fixed
element preserve arbitrary joins and hence are Scott continuous.

An open subset $U$ of a topological space $X$ is \emph{regular open} if
\[
U=\Int{X}{\Cl{X}{U}}.
\]
The set $\RO(X)$ of all regular open subsets, ordered by inclusion, is a
complete Boolean algebra; see, for example,~\cite{BennettDuntsch2007}.  For
$\mathcal A\subseteq\RO(X)$,
\begin{align*}
	\bigvee\mathcal A
	&=\Int{X}{\Cl{X}{\bigcup\mathcal A}},\\
	\bigwedge\mathcal A
	&=\Int{X}{\bigcap\mathcal A},\\
	\neg U&=\Int{X}{X\setminus U}.
\end{align*}
Note that finite meets are ordinary intersections.
We shall also use repeatedly the fact that if $O$ is open in $X$, then
$R:=\Int{X}{\Cl{X}{O}}$ is regular open.  Indeed,
$O\subseteq R\subseteq\Cl{X}{O}$, so $O$ and $R$ have the same closure and
therefore $\Int{X}{\Cl{X}{R}}=R$.

\section{Regular open algebras, principal ideals, and retracts}
	
	In this section, we establish the dense-subspace invariance of $\RO(X)$
	and describe regular open algebras of regular open subspaces as principal
	ideals and Scott-continuous retracts.

\begin{lemma}\label{lem:dense}
If $Y$ is a dense subspace of $X$, then
\[
  \rho:\RO(X)\longrightarrow\RO(Y),\qquad \rho(U)=U\cap Y,
\]
is a complete Boolean algebra isomorphism.  Its inverse is
\[
  \eta:\RO(Y)\longrightarrow\RO(X),\qquad
  \eta(V)=\Int{X}{\Cl{X}{V}}.
\]
\end{lemma}

\begin{proof}
We first record two consequences of density.

If $F$ is closed in $X$, then
\begin{equation}\label{eq:dense-interior}
  \Int{Y}{F\cap Y}=Y\cap\Int{X}{F}.
\end{equation}
The inclusion from right to left is immediate.  For the converse, suppose that
$y\in\Int{Y}{F\cap Y}$.  There is an open set $O\subseteq X$ with
$y\in O$ and $O\cap Y\subseteq F$.  If $O\setminus F$ were nonempty, it
would be a nonempty open set and would meet the dense set $Y$, contradicting
$O\cap Y\subseteq F$.  Hence $O\subseteq F$ and
$y\in Y\cap\Int{X}{F}$.

If $O$ is open in $X$, then
\begin{equation}\label{eq:dense-closure}
  \Cl{X}{O\cap Y}=\Cl{X}{O}.
\end{equation}
Indeed, if $x\in\Cl{X}{O}$ and $N$ is an open neighborhood of $x$, then
$N\cap O$ is a nonempty open set and therefore meets $Y$.  Thus $N$ meets $O\cap Y$.

Let $U\in\RO(X)$.  Using~\eqref{eq:dense-interior} and
\eqref{eq:dense-closure},
\[
\begin{aligned}
  \Int{Y}{\Cl{Y}{U\cap Y}}
    &=\Int{Y}{Y\cap\Cl{X}{U\cap Y}}\\
    &=Y\cap\Int{X}{\Cl{X}{U}}
     =U\cap Y.
\end{aligned}
\]
Thus $\rho(U)\in\RO(Y)$.

Now let $V\in\RO(Y)$.  Choose an open set $O\subseteq X$ with
$V=O\cap Y$.  By~\eqref{eq:dense-closure},
$\eta(V)=\Int{X}{\Cl{X}{V}}=\Int{X}{\Cl{X}{O}}$, which is regular open.
Equation
\eqref{eq:dense-interior} gives
\[\begin{aligned}
\rho(\eta(V)) 
&= \eta(V)\cap Y
  =Y\cap\Int{X}{\Cl{X}{V}}\\
  &=\Int{Y}{Y\cap \Cl{X}{V}}
  =\Int{Y}{\Cl{Y}{V}}
  =V.\end{aligned}
\]
Finally,~\eqref{eq:dense-closure} yields
\[
  \eta(\rho(U))
  =\Int{X}{\Cl{X}{U\cap Y}}
  =\Int{X}{\Cl{X}{U}}
  =U.
\]
Therefore $\rho$ and $\eta$ are inverse order isomorphisms.  An order isomorphism between complete Boolean algebras preserves all joins, meets, and complements.
\end{proof}

\begin{corollary}\label{cor:dense-sobriety}
	For a topological space $X$, the following statements are equivalent:
	\begin{enumerate}[label=\textup{(\arabic*)}]
		\item $\Scott\RO(X)$ is sober.
		\item For every dense subspace $Y$ of $X$, $\Scott\RO(Y)$ is sober.
		\item There exists a dense subspace $Y$ of $X$ such that
		$\Scott\RO(Y)$ is sober.
	\end{enumerate}
\end{corollary}

\begin{proof}
By Lemma~\ref{lem:dense}, every dense subspace $Y$ of $X$ satisfies
$\RO(Y)\cong\RO(X)$, and hence the corresponding Scott spaces are
homeomorphic.  Thus (1) implies (2), while (3) implies (1).  Finally, (2)
implies (3) by taking $Y=X$.
\end{proof}

For a complete Boolean algebra $B$ and $u\in B$, define
\[
  B_{\le u}:=\{b\in B:b\leq u\}.
\]
With the inherited order, $B_{\leq u}$ is a complete Boolean algebra whose
least and greatest elements are $0$ and $u$, respectively.  Its joins and
nonempty meets are computed as in $B$ (the empty meet is $u$), and the
complement of $b\leq u$ relative to $u$ is $u\wedge\neg b$.

We shall also need the following elementary principal-ideal identification.

\begin{lemma}\label{lem:principal}
	Let $U\in\RO(X)$ be endowed with the subspace topology, and define
	\[
	\RO(X)_{\le U}
	:=\{V\in\RO(X):V\subseteq U\}.
	\]
	Then, as collections of subsets of $U$,
	\[
	\RO(U)=\RO(X)_{\le U}.
	\]
	In particular, the identity map is a complete Boolean algebra isomorphism.
\end{lemma}

\begin{proof}
Since $U$ is open in $X$, every open subset of $U$ is open in $X$, and for
every $V\subseteq U$ one has
\[
  \Int{U}{\Cl{U}{V}}
  =U\cap\Int{X}{\Cl{X}{V}}.
\]
Indeed, $\Cl{U}{V}=U\cap\Cl{X}{V}$, and relative interior in the open
subspace $U$ is obtained by intersecting with $U$.  If
$V\in\RO(U)$, then $V\subseteq U$ and
\[
  \Int{X}{\Cl{X}{V}}\subseteq\Int{X}{\Cl{X}{U}}=U.
\]
Also
\[
  V=\Int{U}{\Cl{U}{V}}
   =U\cap\Int{X}{\Cl{X}{V}}
   =\Int{X}{\Cl{X}{V}},
\]
so $V\in\RO(X)$ and $V\subseteq U$.

Conversely, if $V\in\RO(X)$ and $V\subseteq U$, then
\[
  \Int{U}{\Cl{U}{V}}
  =U\cap\Int{X}{\Cl{X}{V}}
  =U\cap V
  =V.
\]
Thus the two ordered collections coincide.  Since the identity is an order
isomorphism between complete Boolean algebras, it preserves all Boolean
operations.
\end{proof}

\begin{lemma}\label{lem:principal-retract}
	Let $B$ be a complete Boolean algebra and let $u\in B$. Then
	$\Scott(B_{\le u})$ is a topological retract of $\Scott B$.
\end{lemma}

\begin{proof}
	Let
	\[
	i\colon B_{\le u}\longrightarrow B
	\]
	be the inclusion map. Since directed suprema in $B_{\le u}$ agree with
	those computed in $B$, the map $i$ is Scott continuous.
	
	Define
	\[
	r\colon B\longrightarrow B_{\le u},
	\qquad
	r(b)=b\wedge u.
	\]
	For every $A\subseteq B$, meeting with $u$ preserves arbitrary joins,
	because $u\wedge(-)$ is a left adjoint to $\neg u\vee(-)$.  Hence
	\[
	r\left(\bigvee A\right)
	=u\wedge\bigvee A
	=\bigvee_{a\in A}(u\wedge a)
	=\bigvee_{a\in A}r(a).
	\]
	Thus $r$ preserves arbitrary joins and is therefore Scott continuous.
	Moreover, for every $b\in B_{\le u}$,
	\[
	r(i(b))=b\wedge u=b.
	\]
	Hence $r\circ i=\operatorname{id}_{B_{\le u}}$, so the induced continuous
	maps exhibit $\Scott(B_{\le u})$ as a topological retract of $\Scott B$.
\end{proof}

\begin{proposition}\label{prop:localtransfer}
	Let $B$ be a complete Boolean algebra and let $u\in B$. If $\Scott B$ is
	sober, then $\Scott(B_{\le u})$ is sober.
\end{proposition}

\begin{proof}
	By Lemma~\ref{lem:principal-retract}, $\Scott(B_{\le u})$ is a topological
	retract of $\Scott B$. Therefore, if $\Scott B$ is sober, then
	$\Scott(B_{\le u})$ is sober by Remark~\ref{rem:standard-facts}.
\end{proof}

\section{A characterization of sobriety for regular open algebras}

In this section, for a second-countable $T_3$ space $X$, we characterize
the sobriety of the Scott space $\Scott\RO(X)$ in terms of the density of
the isolated points of $X$.

\subsection{Countable vanishing bases and complementary pairs}

We first introduce the countable vanishing-base property.  Its abstract form
is essential for the subsequent topological reduction.

\begin{definition}
A complete Boolean algebra $B$ has the \emph{countable vanishing-base
property} if there is a sequence $(g_n)_{n\in\mathbb N}$ in
$B\setminus\{0\}$ such that:
\begin{enumerate}[label=\textup{(\roman*)},leftmargin=2.8em]
  \item for every $b>0$ there is an $n$ with $g_n\leq b$;
  \item for every $n$ there is a decreasing sequence
  $(a_{n,j})_{j\in\mathbb N}$ of nonzero elements below $g_n$ with
  $\bigwedge_j a_{n,j}=0$.
\end{enumerate}
\end{definition}

\begin{lemma}\label{lem:vanishing-ro}
Let $Y$ be a nonempty second-countable $T_3$ space.  Then $\RO(Y)$ has the
countable vanishing-base property if and only if $Y$ has no isolated points.
\end{lemma}

\begin{proof}
Suppose first that $Y$ has no isolated points.  By the Urysohn metrization
theorem~\cite[Chapter~I, Section~9]{Bredon1993}, $Y$ is metrizable; fix a
compatible metric $d$ on $Y$.  By second-countability, there is a countable
base $(E_n)_{n\in\mathbb N}$ consisting of nonempty open sets.
Define
\[
  G_n:=\Int{Y}{\Cl{Y}{E_n}}.
\]
Since $E_n\subseteq G_n\subseteq\Cl{Y}{E_n}$, the sets $E_n$ and $G_n$
have the same closure.  Hence
\[
  \Int{Y}{\Cl{Y}{G_n}}=G_n,
\]
so $G_n$ is a nonzero member of $\RO(Y)$.

We next verify condition~\textup{(i)}.  Let $\emptyset\neq W\in\RO(Y)$ and choose
$y\in W$.  There is a $\delta>0$ such that
\[
  B_d(y,2\delta)\subseteq W.
\]
Choose $n$ with $y\in E_n\subseteq B_d(y,\delta)$.  Since the closure of
$B_d(y,\delta)$ is contained in the closed ball of radius $\delta$, we have
\[
  \emptyset\neq G_n
  \subseteq\Cl{Y}{E_n}
  \subseteq B_d(y,2\delta)
  \subseteq W.
\]

For each $n$, choose $x_n\in G_n$ and $\varepsilon_n>0$ such that
$B_d(x_n,2\varepsilon_n)\subseteq G_n$.  For $j\in\mathbb N$, put
\[
  A_{n,j}
  :=\Int{Y}{\Cl{Y}{B_d(x_n,2^{-j}\varepsilon_n)}}.
\]
As above, every $A_{n,j}$ is regular open.  The sequence
$(A_{n,j})_{j\in\mathbb N}$ is decreasing, and
\[
  \emptyset\neq A_{n,j}\subseteq G_n
  \qquad (j\in\mathbb N).
\]
Indeed, the closure of $B_d(x_n,2^{-j}\varepsilon_n)$ is contained in the
closed ball of the same radius, which is contained in
$B_d(x_n,2\varepsilon_n)$.  Moreover,
\[
  \bigcap_{j\in\mathbb N}A_{n,j}=\{x_n\}.
\]
Indeed, $x_n\in A_{n,j}$ for every $j$, whereas
$A_{n,j}\subseteq\overline{B_d(x_n,2^{-j}\varepsilon_n)}$; every point
different from $x_n$ is excluded by the latter closed ball for all
sufficiently large $j$.
Since $x_n$ is not isolated, $\Int{Y}{\{x_n\}}=\emptyset$.  The formula for
arbitrary meets in $\RO(Y)$ therefore gives
\[
  \bigwedge_{j\in\mathbb N}A_{n,j}
  =\Int{Y}{\bigcap_{j\in\mathbb N}A_{n,j}}
  =\emptyset.
\]
Thus condition~\textup{(ii)} also holds.

Conversely, suppose that $\RO(Y)$ has the countable vanishing-base property
and that $y\in Y$ is isolated.  Since $Y$ is $T_1$, the set
$\{y\}$ is clopen and belongs to $\RO(Y)$.  It has no nonzero proper element
below it.  Condition~\textup{(i)} gives an $n$ such that
\[
  \emptyset\neq G_n\leq\{y\}.
\]
Therefore $G_n=\{y\}$.  Every nonzero element below $G_n$ equals $G_n$, so
the sequence in condition~\textup{(ii)} must satisfy
\[
  \bigwedge_j A_{n,j}=G_n\neq0,
\]
a contradiction.  Hence $Y$ has no isolated points.
\end{proof}

A nonempty subset $F$ of a poset $P$ is \emph{filtered} if for all $x,y\in F$
there is a $z\in F$ with $z\leq x$ and $z\leq y$.  A subset $N$ of  $P$ is \emph{dual Scott open} if 
(i) $N=\mathord{\downarrow}N$ and (ii) for each filtered subset $F$ of $P$ with $\bigwedge F$ existing, $\bigwedge F\in N$ implies $F\cap
N\neq\emptyset$.
Now let $B$ be a complete Boolean algebra.
Since Boolean complementation is an order anti-isomorphism, a subset
$U\subseteq B$ is Scott open if and only if
\[
\neg U:=\{\neg u:u\in U\}
\]
is dual Scott open.

The following complementary-pair theorem drives the order-theoretic
obstruction.

\begin{theorem}\label{thm:mask}
Let $B$ be a complete Boolean algebra with the countable vanishing-base
property.  If $U,V\subseteq B$ are nonempty Scott-open sets, then there
exist $u\in U$ and $v\in V$ such that
\[
  u\wedge v=0,
  \qquad
  u\vee v=1.
\]
\end{theorem}

\begin{proof}
Let $(g_n)_{n\in\mathbb N}$ and
$(a_{n,j})_{n,j\in\mathbb N}$ witness the countable vanishing-base property.
Since a nonempty Scott-open set is an upper set, it contains $1$.  Therefore
\[
  N:=\{\neg u:u\in U\},
  \qquad
  M:=\{\neg v:v\in V\}
\]
are dual Scott open neighborhoods of $0$.

Construct recursively an increasing sequence $(b_n)$ in $N$.  Put $b_{-1}=0$.
Suppose $b_{n-1}\in N$.  The sequence
$(b_{n-1}\vee a_{n,j})_{j\in\mathbb N}$ is decreasing and therefore
filtered.  In a complete Boolean algebra, it holds that
\[
  \bigwedge_j(b_{n-1}\vee a_{n,j})
  =b_{n-1}\vee\bigwedge_j a_{n,j}
  =b_{n-1}\in N.
\]
Because $N$ is dual Scott open and contains this infimum, there exists $j_n$
such that
\[
  b_n:=b_{n-1}\vee a_{n,j_n}\in N.
\]

We claim that $\bigvee_n b_n=1$.  Let $b=\bigvee_n b_n$.  If $b<1$, then
$\neg b>0$.  By condition~\textup{(i)}, there is an index $m$ such that
$g_m\leq\neg b$.  But
\[
  0<a_{m,j_m}\leq g_m\leq\neg b,
  \qquad
  a_{m,j_m}\leq b_m\leq b,
\]
which implies $a_{m,j_m}\leq b\wedge\neg b=0$, a contradiction.

Thus $\bigwedge_n\neg b_n=0$.  The sequence $(\neg b_n)_{n\in\mathbb N}$
is decreasing and therefore filtered.  Since $M$ is dual Scott open and
contains $0$, there exists $n$ such that $\neg b_n\in M$.  The relations
\[
  b_n\in N=\neg U,
  \qquad
  \neg b_n\in M=\neg V
\]
mean precisely that $\neg b_n\in U$ and $b_n\in V$.  Take
$u=\neg b_n$ and $v=b_n$.
\end{proof}

\subsection{An irreducible Scott-closed witness}

We now construct the irreducible Scott-closed set.  The next three lemmas
establish its required properties.

For a complete Boolean algebra $B$, give $B^2$ the coordinatewise order and define
\[
  \Delta_B:=\{(p,q)\in B^2:p\wedge q=0\}.
\]
In particular, $(0,0)\in\Delta_B$, so $\Delta_B$ is nonempty.

\begin{lemma}\label{lem:deltaclosed}
$\Delta_B$ is Scott closed in the product poset $B^2$.
\end{lemma}

\begin{proof}
It is a lower set.  Let
$D=\{(p_i,q_i):i\in I\}\subseteq\Delta_B$ be directed.  For any $i,j$, choose $k$ with
$(p_i,q_i),(p_j,q_j)\le(p_k,q_k)$.  Then
$p_i\wedge q_j\le p_k\wedge q_k=0$.  Since meeting with a fixed element
preserves arbitrary joins in a complete Boolean algebra,
\[
  \left(\bigvee_i p_i\right)\wedge
  \left(\bigvee_j q_j\right)
  =\bigvee_{i,j}(p_i\wedge q_j)=0.
\]
Thus the directed supremum belongs to $\Delta_B$.
\end{proof}

\begin{lemma}\label{lem:deltairreducible}
Suppose every two nonempty Scott-open subsets of $B$ contain disjoint elements.  Then $\Delta_B$ is irreducible in $\Scott(B^2)$.
\end{lemma}

\begin{proof}
Let $O_0,O_1$ be Scott-open subsets of the product poset $B^2$, each meeting $\Delta_B$.  Choose
\[
  x=(p,q)\in O_0\cap\Delta_B,
  \qquad
  y=(r,s)\in O_1\cap\Delta_B.
\]
Thus $p\wedge q=0$ and $r\wedge s=0$.

For $t\in B$, define
\[
  f_x(t)=(p\wedge t,q\wedge t),
  \qquad
  f_y(t)=(r\wedge t,s\wedge t).
\]
Both maps preserve arbitrary joins coordinatewise, hence are Scott continuous.  Therefore
\[
  U=f_x^{-1}(O_0),
  \qquad
  V=f_y^{-1}(O_1)
\]
are nonempty Scott-open subsets of $B$, since both contain $1$.  Choose
$u\in U$ and $v\in V$ with $u\wedge v=0$.

Set
\[
  z=f_x(u)\vee f_y(v)
   =\bigl((p\wedge u)\vee(r\wedge v),
           (q\wedge u)\vee(s\wedge v)\bigr).
\]
Because Scott-open sets are upper sets,
$z\ge f_x(u)\in O_0$ and $z\ge f_y(v)\in O_1$, so
$z\in O_0\cap O_1$.  Expanding the meet of its coordinates gives four terms:
\[
 z_1\wedge z_2
 = (p\wedge q\wedge u)
  \vee(r\wedge s\wedge v)
\vee(p\wedge s\wedge u\wedge v)
  \vee(r\wedge q\wedge u\wedge v)=0.
\]
Hence $z\in\Delta_B$.  Thus every two open sets meeting $\Delta_B$ have an intersection that also meets $\Delta_B$, which is the open-set criterion for irreducibility.
\end{proof}

\begin{lemma}\label{lem:deltanotpoint}
If $B$ is nontrivial, then $\Delta_B$ is not the closure of a point of $\Scott(B^2)$.
\end{lemma}

\begin{proof}
Both $(1,0)$ and $(0,1)$ belong to $\Delta_B$, while $(1,1)$ does not.
Any element of $B^2$ dominating both $(1,0)$ and $(0,1)$ must equal
$(1,1)$.  Hence $\Delta_B$ has no largest element.  By
Remark~\ref{rem:standard-facts}, every point closure is a principal ideal
and therefore has a largest element.
\end{proof}

\begin{lemma}\label{lem:selfsum}
Let $D$ be a nonempty countable metrizable space without isolated points.
Then
\[
  D\sqcup D\cong D.
\]
\end{lemma}

\begin{proof}
By the Sierpi\'nski theorem, $D$ is homeomorphic to $\mathbb Q$, where
$\mathbb Q$ denotes the set of rational numbers equipped with the subspace
topology inherited from $\mathbb R$; see~\cite{Dashiell2021} for a short
self-contained proof.

The topological sum $D\sqcup D$ is again nonempty and countable.  It is
metrizable: if $d$ is a compatible metric on $D$, put
$d_1(x,y)=\min\{d(x,y),1\}$.  Identifying $D\sqcup D$ with
$D\times\{0,1\}$, define
\[
  \rho\bigl((x,i),(y,j)\bigr)
  :=
  \begin{cases}
    d_1(x,y), & i=j,\\
    1,        & i\neq j.
  \end{cases}
\]
This metric induces the topological-sum topology.  Moreover, neither clopen
copy of $D$ has an isolated point, so $D\sqcup D$ has no isolated points.
The Sierpi\'nski theorem therefore also gives $D\sqcup D\cong\mathbb Q$.
Consequently, $D\sqcup D\cong D$.
\end{proof}

\begin{lemma}\label{lem:selfsquare}
Let $X$ be a nonempty second-countable $T_3$ space without isolated points,
and put $B=\RO(X)$.  Then $B\cong B^2$ as complete Boolean algebras.
\end{lemma}

\begin{proof}
By second-countability, let $(E_n)_{n\in\mathbb N}$ be a countable base
consisting of nonempty open sets.  Every $E_n$ is infinite.  Indeed, if a
nonempty finite open set
existed, the $T_1$ property would make each of its points open, contrary to
the assumption that $X$ has no isolated points.

Fix a bijection $e:\mathbb N\to\mathbb N^2$, and write
$e(m)=(n_m,k_m)$.  Suppose that, for some $m\in\mathbb N$, the points
\[
  d_{n_j,k_j}\in E_{n_j}
  \qquad (j<m)
\]
have already been chosen and are pairwise distinct.  Since $E_{n_m}$ is
infinite whereas $\{d_{n_j,k_j}:j<m\}$ is finite, we may choose
\[
  d_{n_m,k_m}
  \in E_{n_m}\setminus\{d_{n_j,k_j}:j<m\}.
\]
Starting with $m=0$ and continuing recursively defines pairwise distinct
points
\[
  d_{n,k}\in E_n
  \qquad (n,k\in\mathbb N).
\]
Put $D=\{d_{n,k}:n,k\in\mathbb N\}$.  The set $D$ is countable and dense
in $X$: every nonempty open subset of $X$ contains some basic open set
$E_n$, and hence contains $d_{n,0}$.  It has no isolated points: if $d\in D$
and $O$ is an open
neighborhood of $d$ in $X$, choose $n$ with $d\in E_n\subseteq O$; then
$D\cap O$ contains the infinite set $\{d_{n,k}:k\in\mathbb N\}$.

By the Urysohn metrization
theorem~\cite[Chapter~I, Section~9]{Bredon1993}, $X$ is metrizable, and hence
its subspace $D$ is metrizable.
Thus $D$ is a nonempty countable metrizable space without isolated points.
Lemma~\ref{lem:selfsum} therefore yields $D\sqcup D\cong D$.
Every homeomorphism induces a complete Boolean algebra isomorphism between
the corresponding regular open algebras, by taking direct images.

Writing $D_0$ and $D_1$ for the two clopen copies of $D$ in $D\sqcup D$,
the map
\[
  W\longmapsto (W\cap D_0,W\cap D_1)
\]
is, after the canonical identifications of $D_0$ and $D_1$ with $D$, a
complete Boolean algebra isomorphism
\[
  \RO(D\sqcup D)\cong\RO(D)^2.
\]
Its inverse sends $(U_0,U_1)$ to $U_0\cup U_1$.
Indeed, closure and interior are computed separately on the two clopen
summands.
Since $D$ is dense in $X$, Lemma~\ref{lem:dense} now gives
\[
  B^2
  \cong\RO(D)^2
  \cong\RO(D\sqcup D)
  \cong\RO(D)
  \cong B,
\]
completing the proof.
\end{proof}

\begin{theorem}\label{thm:crowdedobstruction}
If $X$ is a nonempty second-countable $T_3$ space without isolated points,
then $\Scott\RO(X)$ is not sober.
\end{theorem}

\begin{proof}
Put $B=\RO(X)$.  Since $X$ is nonempty, $B$ is nontrivial.
By Lemma~\ref{lem:vanishing-ro}, $B$ has the countable vanishing-base
property.  Theorem~\ref{thm:mask} supplies the disjoint elements required in
Lemma~\ref{lem:deltairreducible}.  Hence, by
Lemmas~\ref{lem:deltaclosed}--\ref{lem:deltanotpoint}, $\Delta_B$ is a
nonempty irreducible Scott-closed subset of $B^2$ that is not the closure of
a point.  Thus $\Scott(B^2)$ is not sober.  Lemma~\ref{lem:selfsquare}
gives an order isomorphism $B\cong B^2$, which is a homeomorphism between
their Scott spaces.  Therefore $\Scott B$ is not sober.
\end{proof}

\subsection{The characterization theorem}

\begin{theorem}\label{thm:localobstruction}
Let $X$ be a topological space.  Suppose there exists a nonempty
$U\in\RO(X)$ such that the subspace $U$ is a second-countable $T_3$ space with
no isolated points.  Then $\Scott\RO(X)$ is not sober.
\end{theorem}

\begin{proof}
Theorem~\ref{thm:crowdedobstruction} shows that $\Scott\RO(U)$ is not sober.
By Lemma~\ref{lem:principal},
$\RO(U)$ is the principal ideal $\RO(X)_{\le U}$.
Proposition~\ref{prop:localtransfer} transfers the non-sobriety to
$\Scott\RO(X)$: otherwise sobriety of $\Scott\RO(X)$ would imply sobriety of
its retract $\Scott\RO(U)$.
\end{proof}

We first relate atomicity of the regular open algebra to isolated points.

Let $B$ be a Boolean algebra. A nonzero element $a\in B$ is an
\emph{atom} if the only element $b\in B$ satisfying $0<b\leq a$ is $a$
itself.  The Boolean algebra $B$ is \emph{atomic} if every nonzero element
of $B$ has an atom below it.

\begin{lemma}\label{lem:atomic-ro}
Let $X$ be a $T_3$ space.  Then the atoms of $\RO(X)$ are precisely the
singletons $\{x\}$ with $x\in\Iso(X)$.  Consequently,
\[
  \RO(X)\text{ is atomic}
  \quad\Longleftrightarrow\quad
  \Iso(X)\text{ is dense in }X.
\]
\end{lemma}

\begin{proof}
If $x\in\Iso(X)$, then $\{x\}$ is open and, since $X$ is $T_1$, also
closed.  Hence $\{x\}\in\RO(X)$, and it is an atom.
Conversely, let $A$ be an atom of $\RO(X)$ and choose $x\in A$.  Suppose
that there exists $y\in A\setminus\{x\}$.  Since $X$ is $T_1$, the set
$A\setminus\{y\}$ is an open neighborhood of $x$.  By regularity, there is
an open set $W$ such that
\[
  x\in W\subseteq\Cl{X}{W}\subseteq A\setminus\{y\}.
\]
Put $V=\Int{X}{\Cl{X}{W}}$.  Then $V$ is a nonempty regular open set and
\[
  V\subseteq A\setminus\{y\}\subsetneq A,
\]
contradicting the fact that $A$ is an atom.  Therefore $A=\{x\}$, and hence
$x$ is isolated.

Suppose now that $\Iso(X)$ is dense.  Every nonempty member $U$ of
$\RO(X)$ contains some $x\in\Iso(X)$, and the atom $\{x\}$ lies below $U$.
Thus $\RO(X)$ is atomic.

Conversely, suppose that $\RO(X)$ is atomic, and let $O$ be a nonempty open
subset of $X$.  Choose $x\in O$.  By regularity, there is an open set $W$
such that
\[
  x\in W\subseteq\Cl{X}{W}\subseteq O.
\]
Then $U=\Int{X}{\Cl{X}{W}}$ is a nonempty member of $\RO(X)$ contained in
$O$.  By atomicity, $U$ contains an atom of $\RO(X)$, which by the first
part of the proof is a singleton consisting of an isolated point.  Hence
every nonempty open subset of $X$ meets $\Iso(X)$, so $\Iso(X)$ is dense.
\end{proof}

We are now ready to prove the main theorem of this paper.

\begin{theorem}\label{thm:main}
For a second-countable $T_3$ space $X$, the following statements are
equivalent:
\begin{enumerate}[label=\textup{(\arabic*)},leftmargin=2.8em]
  \item $\Scott\RO(X)$ is sober.
  \item The complete Boolean algebra $\RO(X)$ is atomic.
  \item $\Iso(X)$ is dense in $X$.
\end{enumerate}
\end{theorem}

\begin{proof}
By Lemma~\ref{lem:atomic-ro}, statements~\textup{(2)} and~\textup{(3)} are
equivalent.  It remains to prove the equivalence of~\textup{(1)}
and~\textup{(3)}.

(3) $\Rightarrow$ (1): Put $D=\Iso(X)$.  Suppose first that $D$ is dense.  It is discrete as a
subspace of $X$, and Lemma~\ref{lem:dense} gives
\[
  \RO(X)\cong\RO(D)=\Pow(D).
\]
Remark~\ref{rem:standard-facts} shows that the Scott space of the right-hand side
is sober.  Hence $\Scott\RO(X)$ is sober.

(1) $\Rightarrow$ (3): Suppose on the contrary that $D=\Iso(X)$ is not dense.  Then
\[
  O:=X\setminus\Cl{X}{D}
\]
is a nonempty open set.  Choose $x\in O$.  By regularity, choose a nonempty
open set $W$ such that
\[
  x\in W\subseteq\Cl{X}{W}\subseteq O.
\]
Set
\[
  U:=\Int{X}{\Cl{X}{W}}.
\]
Since  $W$ is open, it follows that $U$ is regular open.  Moreover,
$U$ is nonempty and $U\subseteq O$.  Since second-countability and the
$T_3$ property are hereditary to subspaces, $U$ is a second-countable
$T_3$ space.

Finally, $U$ has no isolated points.  If $u\in U$ were isolated in $U$, then
$\{u\}$ would be open in $X$ because $U$ is open in $X$.  Thus $u\in D$, contradicting
$U\subseteq O=X\setminus\Cl{X}{D}$.

By Theorem~\ref{thm:localobstruction}, $\Scott\RO(X)$ is not sober,
which completes the proof.
\end{proof}

\begin{remark}\label{rem:metric}
	The implication from the density of $\Iso(X)$ to the sobriety of
	$\Scott\RO(X)$ holds for every topological space; neither
	second-countability nor the $T_3$ assumption is needed in that direction.
	By the Urysohn metrization
	theorem~\cite[Chapter~I, Section~9]{Bredon1993}, every second-countable $T_3$ space
	is metrizable.  Thus the use of compatible metrics in
	Lemmas~\ref{lem:vanishing-ro} and~\ref{lem:selfsquare} imposes no additional
	assumption.  In particular, Theorem~\ref{thm:main} applies to every
	second-countable Hausdorff topological manifold.
\end{remark}

Let $2^\mathbb{N}$ denote the Cantor space, i.e.,
\[
2^{\mathbb{N}}=\prod_{n\in\mathbb{N}}\{0,1\},
\]
equipped with the product topology, where  $\{0,1\}$ is endowed with the discrete topology. 

For every positive integer $n$, the Euclidean space $\mathbb R^n$ and the
Cantor space are second-countable $T_3$ spaces without isolated points.
The following corollary therefore follows immediately from
Theorem~\ref{thm:main}.

\begin{corollary}\label{cor:cantor-space}
	\begin{enumerate}[label=\textup{(\arabic*)}]
		\item For every positive integer $n$, the Scott space
		$\Scott\bigl(\RO(\mathbb{R}^n)\bigr)$ is not sober.
		\item The Scott space $\Scott\bigl(\RO(2^{\mathbb{N}})\bigr)$ is not sober.
	\end{enumerate}
\end{corollary}

%
%
%

\vspace{0.7cm}

\noindent{\bf Acknowledgments: }
The second author is supported by the National Natural Science Foundation of China (No. 12571507, 12401607, 12101313), the Basic Research Program of Jiangsu Province (No. BK20241086, BK20241087).

\vspace{0.7cm}

\noindent{\bf Declaration of interests: }
The authors declare that they have no known competing financial interests or personal relationships that could have appeared to influence the work reported in this paper.

\section*{Declaration of generative AI and AI-assisted technologies
	in the manuscript preparation process}
During the preparation of this work, the authors used ChatGPT (OpenAI)
to assist with language refinement, manuscript organization, and literature
review. After using this tool, the authors reviewed and edited the content
as needed and take full responsibility for the content of the published article.

\end{document}